\documentclass{amsart}
\usepackage{amsfonts,amsmath, comment, amssymb, rotating, color,  graphicx}
\newtheorem{theorem}{Theorem}
\theoremstyle{plain}

\newtheorem{corollary}{Corollary}
\newtheorem{definition}{Definition}

\newtheorem{lemma}{Lemma}

\newtheorem{remark}{Remark}
\numberwithin{equation}{section}

\newcommand{\F}{{\mathbb F}}

\begin{document}
\title[Cycles, wheels, and gears in finite planes]{Cycles, wheels, and gears in finite planes}
\author{Jamie Peabody}
\address{Department of Mathematics\\
California State University, Fresno \\ 
Fresno, CA 93740.}
\email{jpeabody09@mail.fresnostate.edu}
\author{Oscar Vega}
\address{Department of Mathematics\\
California State University, Fresno \\ 
Fresno, CA 93740.}
\email{ovega@csufresno.edu}
\author{Jordan White}
\address{Department of Mathematics and Statistics\\
California State University, Monterey Bay \\ 
Seaside, CA 93955.}
\email{jorwhite@csumb.edu}
\thanks{This work was supported by NSF Grant \#DMS-1156273 (Fresno State's 2012 Summer REU), and the McNair Program at Fresno State}
\date{}

\subjclass[2010]{Primary 05; Secondary 51}

\begin{abstract}
The existence of a primitive element of $GF(q)$ with certain properties is used to prove that all cycles that could theoretically be embedded in $AG(2,q)$ and $PG(2,q)$ can, in fact, be embedded there (i.e. these planes are `pancyclic'). We also study embeddings of wheel and gear graphs in arbitrary projective planes.
\end{abstract}

\maketitle

\section{Introduction}

In this article, a graph will be understood to be simple, finite, and undirected.  Since we will mostly focus on cycles and cycle-related graphs we define, for $k \ge 3$, a $k$-cycle as the graph with $V = \{x_1,\ldots ,x_k\}$ and $E = \{x_1x_2, x_2x_3, \ldots ,x_{k-1}x_k, x_kx_1 \}$. We denote this graph by $C_k$. We refer the reader to  \cite{W} for any graph theoretical notion we use and fail to define.

Next we define the concepts in finite geometry that we will need later on, any missing concepts may be found in \cite{Demb}.

\begin{definition}
Let $\pi = ( \mathcal{P}, \mathcal{L}, \mathcal{I} )$ where $\mathcal{P}$ is a set of points, $\mathcal{L}$ is a set of lines, and $\mathcal{I}$ is an incidence relation. Then $\pi$ is an affine plane if it satisfies the following conditions:  \\
\textbf{(a)}  Given any two distinct points, there is exactly one line incident with both of them. \\
\textbf{(b)}  For every line $l$ and and every point $P$ not incident with $l$ there is a unique line $m$ that is incident with $P$ and that does not intersect $l$. \\
\textbf{(c)}  There are three points that not on the same line.
\end{definition}

We may obtain a projective plane from the affine plane by the addition of a line at infinity, denoted $\ell_{\infty}$. Furthermore, lines which were parallel with one another in the affine plane, meet at a point at infinity in the projective planes. Lastly, these points at infinity are all incident with the line at infinity. Conversely, deleting any line in a projective plane (and all points incident with that line) yields an affine plane.

In this work we consider projective planes that contain only a finite number of points and lines.  In this case, it is known that for every affine plane $\pi = ( \mathcal{P}, \mathcal{L}, \mathcal{I} )$ there is a positive integer $q$, called the order of the plane, such that $|\mathcal{P}| = q^2$, $|\mathcal{L}| = q^2+q$, each line contains exactly $q$ points, and  every point is incident with exactly $q+1$ lines. A similar result is also valid for projective planes. In this case the addition of $\ell_{\infty}$ yields $|\mathcal{P}| =|\mathcal{L}|= q^2+q+1$, that every line contains $q+1$ points, and that every point is incident with $q+1$ lines. All known examples of finite planes have order equal to the power of a prime number.

When a finite affine/projective plane may be coordinatized by a field we say that the plane is Desarguesian. It is known that there is only one, up to isomorphism, Desarguesian plane of each order. We denote the affine plane of order $q$ by $AG(2,q)$ and the projective one by $PG(2,q)$. It follows that we may `model' $AG(2,q)$ by using $GF(q) \times GF(q)$, where $GF(q)$ is the finite field with $q$ elements. \\

Our objective in this article is to study how cycles, and some cycle-related graphs, can be embedded in finite planes (both affine and projective). In order to do this we must define what we will understand for an embedding of a graph into a finite plane.

\begin{definition}
Let $G=(V,E)$ be a graph. An embedding of $G$ into a plane (affine or projective)  $\pi=(\mathcal{P}, \mathcal{L}, \mathcal{I})$, is an injective function $\psi:V\rightarrow \mathcal{P}$ that induces naturally an (also) injective function $\overline{\psi}:E\rightarrow \mathcal{L}$ by preserving incidence. \\
We call $\psi$ an \emph{embedding of $G$ in $\pi$}. If such a function  exists, we say that \emph{G embeds in $\pi$} and write $G\hookrightarrow\pi$.
\end{definition}

Note that since $\overline{\psi}$ is injective, then we will identify edges in $G$ with \emph{whole} lines. That is, if a line has been used as an edge for a graph, this line cannot be used again in the same embedding.

\begin{definition}
We say that $AG(2,q)$ is pancyclic if and only if $C_{k}\hookrightarrow AG(2,q)$, for all $3\leq k\leq q^2$. Similarly, we say that $PG(2,q)$ is pancyclic if and only if $C_{k}\hookrightarrow PG(2,q)$, for all $3\leq k \leq q^2+q+1$.
\end{definition}

\bigbreak

The idea of embedding a graph into other structures has been present for a long time. For instance, the history of embeddings of graphs into linear spaces goes back to Hall \cite{Hall43}, includes Erd\"os \cite{Erd79}, and the more recent work by Moorhouse and Williford \cite{MoorWil2009}. On the other hand, not much is known about embeddings of graphs in finite planes, and most of what is known is on embeddings of cycles. This is, most probably, because studying $k$-cycles embedded in a projective plane $\pi$ is equivalent to study embeddings of $(2k)$-cycles in $Levi(\pi)$.  For instance, one can use the Singer cycle in $PG(2,q)$ to construct a $(q^2+q+1)$-cycle in $PG(2,q)$ (e.g. see \cite{LMV09}). Also, the constructions by  Schmeichel \cite{Schm89} proved that $PG(2,p)$ is pancyclic for $p$ prime. Moreover, Schmeichel's `longest' cycle is different from the one constructed using the Singer cycle. More recently, in \cite{LMV09} one may find expressions for the number of $k$-cycles in a projective plane of order $q$, for $3\leq k \leq 6$. This work has been extended by Voropaev \cite{Vor12} to $7\leq k \leq 10$.

Our work may also be related to \cite{LMV12}, as in that article embeddings of cycles in projective planes are also studied. However, the approach in \cite{LMV12} is purely geometrical, and our approach in this paper is an effort to bring an algebraic perspective to the problem. In fact, we will coordinatize $AG(2,q)$ using a field to give an algebraic characterization of the pancyclicity of $AG(2,q)$ and $PG(2,q)$. We refer the reader to \cite{LMV09} and  \cite{LMV12} for a thorough historical narrative on this problem.

\section{Cycles in $AG(2,q)$ and $PG(2,q)$}

In this section we investigate pancyclicity in $AG(2,q)$ and $PG(2,q)$ by modifying the construction of cycles in \cite{LMV12}.

Let $\F=GF(q)$ and let $<\alpha> =\F^*$, that is, $\alpha$ is a primitive element of $\F$. In this section we denote $AG(2,q)$ by $\pi$.

We will consider the following  coordinatization of $\pi$ by using $\F\times \F$. The points on its `axes' will be labeled using $0$ or powers of $\alpha$. Next we label the lines  through $\mathcal{O}=(0,0)$ as follows:
\[
l_0 : x=0 \hspace{.9in} l_i: y=x\alpha^i , \ \ \text{for} \ i = 1, \cdots , q-1 \hspace{.9in} l_q: y=0
\]
Also, for any point $Q$ in the plane we denote the line parallel to $l_i$ that passes through $Q$ by $l_{i}+Q$.

Pick any point $P_0\in l_0$, different from $\mathcal{O}$. We define $P_i= (l_{i+1}+P_{i-1})\cap l_i$, for all $i=1, \cdots , q-1$. Similarly, $P_q= (l_{0}+P_{q-1})\cap l_q$. Next we connect $P_{i-1}$ with $P_{i}$ using $l_{i+1}+P_{i-1}$, for all $i=1, \cdots , q-1$, and connect $P_{q-1}$ with $P_q$ using $l_{0}+P_{q-1}$. In this way we obtain a path of length $q+1$. We denote this path by $\mathcal{P}_{P_0}$.

In \cite{LMV12} it is shown that the $q-1$ paths constructed this way, with $P_0$ being any point on $l_0$ different from $\mathcal{O}$ share no points or lines. Hence, these paths partition the points of $\pi \setminus \{\mathcal{O}\}$. Moreover, a path starting at $(0, \beta)$ may be obtained from the path starting at $(0, \alpha)$ by using a translation $T_v$ with $v=(0, \beta -\alpha)$.

Note that no line parallel to $l_1$ has been used in the construction of these paths. Hence, using the line $l_{1}+P_{q}$ we connect $P_q$ with a (uniquely determined) point $Q_0$ on $l_0$. Note that if $Q_0=P_0$ then we get a cycle of length $q+1$. On the other hand, if  $Q_0\neq P_0$ then we may `glue' $\mathcal{P}_{Q_0}$ to the path starting at $P_0$ and ending at $Q_0$ we have just built, yielding a longer path. It seems that  when $P_0\neq Q_0$ we are able to create `long' cycles. But, how long?  In order to answer this question we need to study when $P_0\neq Q_0$. We will do this algebraically, by identifying $\pi$ with $\F\times \F$.  

\begin{lemma} \label{lem1}
Let $P_0 = (0, \beta) \in \pi$ then 
\[
P_{i+1} \ = \ \ (y = \alpha^{i+2}x+\beta(1+\alpha)^i)\cap(y=\alpha^{i+1}x) \ \ = \ \ \left(\dfrac{\beta(1+\alpha)^i}{\alpha^{i+1}(1-\alpha)}, \dfrac{\beta(1+\alpha)^i}{1-\alpha}\right)
\]
for all $0 \le i \le (q-2)$. Also,
\[
P_q = \left(\dfrac{\beta}{(1+\alpha)^{2}(1-\alpha)},0 \right)  \hspace{1in}  Q_0 =  \left(0, \dfrac{-\alpha\beta}{(1-\alpha)(1+\alpha)^2}\right) 
\]
\end{lemma}

\begin{proof}
Let $p_0 = (0, \beta)$, then
\[
P_1 = (y = \alpha^2 + \beta)\cap(y = \alpha x) = \left(\dfrac{\beta}{\alpha(1-\alpha)}, \dfrac{\beta}{(1-\alpha)} \right)
\]
In general,
\[
P_{i+1} = (y = \alpha^{i+2}x+b)\cap(y=\alpha^{i+1}x)
\]
We find $b$ by plugging the coordinates of $P_{i}$ into $y = \alpha^{i+2}x+b$. We get:
\[
b \ = \ \ \beta(1+\alpha)^{i-1}\Bigg(\dfrac{\alpha^i-\alpha^{i+2}}{\alpha^i(1-\alpha)}\Bigg)  \ \ = \ \  \beta(1+\alpha)^{i}
\]
So,
\[
P_{i+1} = (y = \alpha^{i+2}x+\beta(1+\alpha^i))\cap(y=\alpha^{i+1}x)
\]

We solve for $x$ to get
\[
x = \dfrac{\beta(1+\alpha^i)}{\alpha^{i+1}(1 - \alpha)}
\]
It follows that, for $1 \le i \le (q-2)$,
\[
P_{i+1} = \left(\dfrac{\beta(1+\alpha)^i}{\alpha^{i+1}(1-\alpha)}, \dfrac{\beta(1+\alpha)^i}{(1-\alpha)}\right)
\]

We obtain what $P_q$ and $Q_0$ look like using similar procedures.
\end{proof}

The previous lemma proves that each of the paths of the form $\mathcal{P}_{P_0}$ starts at $(0, \beta)$ and `comes back' to $l_0$ at $\left(0, \dfrac{-\alpha\beta}{(1-\alpha)(1+\alpha)^2}\right) $. Note that this behavior is being dictated by the action of the group $\F^*$ on itself defined by 
\[
\alpha \cdot \beta   =   \dfrac{-\alpha}{(1-\alpha)(1+\alpha)^2} \ \beta
\]
The following result is almost immediate.

\begin{theorem}\label{thmmain1}
Assume that there is a primitive element $\alpha \in \F$ such that 
\[
\gamma =\dfrac{-\alpha}{(1-\alpha)(1+\alpha)^2}
\]
is also primitive. Then,  $C_{q^2-1} \hookrightarrow  \pi$.
\end{theorem}

\begin{proof}
Having $\gamma$ being primitive means that the `sequential action' of $\alpha$ on $\F^*$ yields the cycle
\[
\beta \xrightarrow{~\alpha~} \gamma \beta \xrightarrow{~\alpha~} \gamma^2 \beta \xrightarrow{~\alpha~} \cdots \xrightarrow{~\alpha~} \gamma^{q-2} \beta \xrightarrow{~\alpha~} \gamma^{q-1}\beta =\beta
\] 
which runs through all the elements in $\F^*$. Hence, the paths of the form $\mathcal{P}_{P_0}$, when connected using lines parallel to $l_1$ create a cycle of length $q^2-1$.
\end{proof}

\begin{remark}\label{rem10million}
We have not been able to prove that the hypothesis in Theorem \ref{thmmain1}  must hold. However, we used \emph{Mathematica} to verify that such hypothesis does hold for all finite fields of order at most $10^6$. We then used \emph{Python} to check that such hypothesis also holds for fields of prime order $10^6\leq p \le 10^7$.
\end{remark}

\begin{remark}\label{remjustaplaceholder}
Different labelings of the lines through $\mathcal{O}$ in $\pi$ might yield different `$\gamma$s' in Theorem \ref{thmmain1}. For instance, 
\[
l_0: x=0 \hspace{.4in}  l_i: y=\alpha^i x, \ \ \text{for} \ 1 \leq i \leq (q-2) \hspace{.4in} l_{q-1}: y=0 \hspace{.4in} l_{q}: y=x
\]
yields
\[
\gamma' = \dfrac{\alpha-1}{(\alpha+1)^3}
\]
We tried many different labelings on the lines through $\mathcal{O}$, no other `interesting' $\gamma$s were found.
\end{remark}

If $\gamma$ obtained in Theorem \ref{thmmain1} were equal to $\gamma'$ in Remark \ref{remjustaplaceholder} we would get $3 \alpha=1$, which would mean that $char(\F)\neq 3$. Moreover, if we also assume $\gamma'=1$ then $(\alpha+1)^3 = \alpha -1$, which implies $\alpha^4  = -1$, and thus that $<\alpha>$ has order $8$, forcing $\F =GF(9)$. But this  is impossible because $char(\F)\neq 3$. Hence, $\gamma \neq 1$ or $\gamma' \neq 1$. It follows that one may choose a coordinatization of $\pi$ such that $P_0\neq Q_0$, and thus that our construction may be always assumed to create cycles of length at least $2(q+1)$.

\begin{lemma}
Let $q=2^a$, where $a\in \mathbb{N}$, $a>1$, and $\gamma'$ is the element obtained in Remark \ref{remjustaplaceholder}. Then, there is a primitive element $\alpha$ of $\F=GF(q)$ such that $\gamma'$ is also a primitive element.
\end{lemma}

\begin{proof}
When $q$ is even we obtain,
\[
\gamma' = \dfrac{1}{(\alpha+1)^2}
\]

One of the Golomb conjectures (a theorem now), says that if $q$ is even and larger than $2$ then there are two consecutive primitive elements, $\alpha$ and $\alpha+1$, of $\F$ (see, for example, the survey \cite{Cohen}). Next, since $q-1$ is odd then $(\alpha+1)^2$ is a primitive element because $\alpha+1$ is a primitive element and $\gcd(2, q-1)=1$.  Finally, we use that the inverse of a primitive element is also a primitive element to get that $ \dfrac{1}{(\alpha+1)^2}$ is a primitive element.
\end{proof}

\begin{remark}
We will say that  \emph{Hypothesis J} holds when $q$ is a power of $2$, or the hypothesis in Theorem \ref{thmmain1} holds. \\
As of now, we know that  \emph{Hypothesis J}  holds when If $q=2^a$, for some $a\in \mathbb{N}$, or  $q$ is an odd prime less than $10^7$, or $q$ is a power of an odd prime that is less than $10^6$.
\end{remark}

The following corollary is immediate. 

\begin{corollary}\label{HJ+Golomb}
If  \emph{Hypothesis J} holds, then $C_{q^2-1} \hookrightarrow  \pi$.
\end{corollary}

The largest possible cycle that could be embedded in $\pi$ has length $q^2$. We will construct such a cycle in $\pi$ by noticing that the $q^2-1$ cycles already constructed do not use any of the lines through $\mathcal{O}$. 

\begin{corollary}
If  \emph{Hypothesis J} holds, then $C_{q^2} \hookrightarrow  \pi$.
\end{corollary}

\begin{proof}
Let $P$ and $Q$ be two points in $\pi$ that are adjacent in the embedding of $C_{q^2-1}$ in $\pi$ described in Corollary \ref{HJ+Golomb}. Assume, WLOG,  that $P\in l_0$ and $Q\in l_1$. We `disconnect' $P$ and $Q$ by eliminating the line that joins them, and then we connect each one of them with $\mathcal{O}$ by using $l_0$ and $l_1$. This new cycle has length $q^2$.
\end{proof}

As of now we have prove the embedding of cycles of length $q^2-1$ and $q^2$. What about shorter cycles?

\begin{theorem}\label{thmAGpan}
If  \emph{Hypothesis J} holds, then $\pi$ is pancyclic.
\end{theorem}

\begin{proof}
We only need to prove $C_{k} \hookrightarrow  \pi$, for all $3\leq k \leq q^2-2$. We know (see \cite{LMV09}) that $K_{q+1} \hookrightarrow  \pi$, and thus we get $C_{k} \hookrightarrow  \pi$, for all $3\leq k \leq q+1$. \\
For $q+2 \leq k \leq q^2-2$ we let 
\[
P_0 \rightarrow P_1 \rightarrow P_2 \rightarrow \cdots \rightarrow P_{q^2-3} \rightarrow P_{q^2-2} \rightarrow P_0
\]
be the $(q^2-1)$-cycle in Corollary \ref{HJ+Golomb}. \\
Let $k-1=(q+1)\lambda+r$, where $r,\lambda \in \mathbb{N}$ and $0\leq r<q+1$. We have two cases: \\
\textbf{(a)} If $r\neq 0$ then the vertices $P_1$ and $P_{k-1}$ are not on the same line through $\mathcal{O}$. Hence, the cycle
\[
\mathcal{O} \rightarrow P_1 \rightarrow P_2 \rightarrow \cdots \rightarrow P_{k-2} \rightarrow P_{k-1} \rightarrow \mathcal{O}
\]
is a $k$-cycle embedded in $\pi$. \\
\textbf{(b)} If $k-1=(q+1)\lambda$ then $\mathcal{O}$ and the vertices $P_1$ and $P_{k-1}$ are collinear. But neither $P_{k-3}$ nor $P_{k-2}$ are on the line joining $P_1$ and $\mathcal{O}$, as $q>1$. Also note that $\mathcal{O}$, $P_{k-2}$, and $P_{(k-3)+(q+1)+1}$ are collinear. \\
Since $k \leq q^2-2$ then $\lambda \leq q-2$, and thus $(k-3)+(q+1) +1 \leq q^2-1$. Hence, the cycle
\[
\mathcal{O} \rightarrow P_1 \rightarrow  \cdots \rightarrow  P_{k-3} \rightarrow P_{(k-3)+(q+1)} \rightarrow P_{(k-3)+(q+1)+1}  \rightarrow \mathcal{O}
\]
is an embedding of $C_k$ in $\pi$.
\end{proof}

We now move on to prove a result equivalent to Theorem \ref{thmAGpan} for projective planes. For the rest of the section we will denote $PG(2,q)$ by $\Pi$. \\

Since $\Pi$ is constructed from $AG(2,q)$ then we know that Theorem \ref{thmAGpan} also holds in $\Pi$. It is also known that $C_{q^2+q+1}$ embeds in $\Pi$, this cycle is constructed from the Singer cycle of the plane (see \cite{LMV09} and \cite{Singer38}). So, in order to get pancyclicity in $\Pi$ we only need to be able to embed $k$-cycles with length $q^2\leq k \leq q^2+q$. Our plan is to modify the embedding of $C_{q^2-1}$ in $AG(2,q)$ described in Corollary \ref{HJ+Golomb}. 

We first take the $(q^2-1)$-cycle embedded in $\Pi$ described in Corollary \ref{HJ+Golomb} and shorten it to get the following path on $q^2-q-1$ vertices. 
\[
\mathcal{P}: \ \ P_{1} \rightarrow P_2   \rightarrow  \cdots  \rightarrow P_{q^2-q-2}  \rightarrow P_{q^2-q-1} 
\]

Note that the $q+1$ affine points $P_{q^2-q}, P_{q^2-q+1}, \cdots, P_{q^2-2}$, and $P_0$ have not been used in this path. Now, since $P_{q^2-q} \in l_2$, $P_{q^2-q+1} \in l_3$, $\cdots$, $P_{q^2-2}\in l_q$, and $P_0\in l_0$, we will re-label these points (to make notation easier later on) as follows
\[
P_{q^2-q} =Q_2 \hspace{.4in} P_{q^2-q+1} = Q_3 \hspace{.1in} \cdots \hspace{.1in}  P_{q^2-2} = Q_q \hspace{.4in}  P_0=Q_0
\]
Hence, $Q_i \in l_i$, for all $i=2,3,\cdots, q, 0$. \\

The points of $\Pi$ not used in this path are: \\
\textbf{(a)} the $q+1$ points on  $\ell_{\infty}$: $\{ (0), (1), \cdots , (q-1), (q)  \}$, where $(i)$ is the point on $\ell_{\infty}$ incident with $l_i$, for all $i=0,1,\cdots, q$. \\
\textbf{(b)} the $q+1$ affine points $Q_{2}, Q_{3}, \cdots, Q_{q}, Q_0$, and $\mathcal{O}$.  \\

In terms of lines, we have not used: \\
\textbf{(a)} $\ell_{\infty}$, \\
\textbf{(b)} the $q+1$ lines through $\mathcal{O}$, \\
\textbf{(c)} The $q-1$ lines $m_i$, joining $Q_i$ and $Q_{i+1 \bmod q+1}$, for all $i= 2, 3, \cdots, q$. Also, the line $m_0$ connecting $Q_0$ with $P_1$, and the line $m$ joining $P_{q^2-q-1}$ and $Q_{2}$. All this yields $q+1$ lines not incident with $\mathcal{O}$. \\

Now we are ready to prove pancyclicity in $\Pi$.

\begin{theorem}
If  \emph{Hypothesis J} holds, then $\Pi$ is pancyclic.
\end{theorem}

\begin{proof}
We will use the path (and information) described above. Recall that all that was based on the cycle described in Corollary \ref{HJ+Golomb}, which needs \emph{Hypothesis J}. Also recall that we only need to construct $k$-cycles, for $q^2\leq k \leq q^2+q$.

Note that $m_i$  is parallel to $l_{i+2 \bmod q+1}$, and thus goes through $(i+2\bmod q+1)$, for all $i= 2, 3, \cdots, q$. Similarly, $m_0$ goes through $(2)$, and the line joining $m$ is incident with $(3)$. Now note that the following two paths
\[
Q_2 \xrightarrow{l_2} (2) \xrightarrow{\ell_{\infty}} (3) \xrightarrow{l_3} Q_3   \xrightarrow{m_2} (4) \xrightarrow{l_4}  \cdots  \xrightarrow{m_{q-1}} (q) \xrightarrow{l_q} Q_{q} \xrightarrow{m_q} Q_0
\]
\[
Q_0 \xrightarrow{m_0} \underbrace{P_{1}   \rightarrow  \cdots  \rightarrow P_{q^2-q-1}}_{in \ \ \mathcal{P}} \xrightarrow{m} Q_2
\]
may be joined to create a $(q^2+q-1)$-cycle in $\Pi$. Let us call this cycle $\mathcal{C}$. \\
Since $l_0$, $l_1$, and $\mathcal{O}$ have not been used in $\mathcal{C}$, a slight modification of it allows us to get a $(q^2+q)$-cycle. That cycle is:
\[
Q_2 \xrightarrow{l_2} (2) \xrightarrow{\ell_{\infty}} (3) \xrightarrow{l_3}   \cdots   \xrightarrow{m_q} Q_0 \xrightarrow{l_0} \mathcal{O} \xrightarrow{l_1} \underbrace{P_{1}   \rightarrow  \cdots  \rightarrow P_{q^2-q-1}}_{in \ \ \mathcal{P}} \xrightarrow{m} Q_2
\]
Hence, we just need to construct $k$-cycles, for $q^2\leq k \leq q^2+q-2$.  \\
First notice that if, in $\mathcal{C}$, instead of the subpath 
\[
P_{q^2-q-1}  \xrightarrow{m}  Q_2 \xrightarrow{l_2} (2) \xrightarrow{\ell_{\infty}} (3)\xrightarrow{l_3} Q_3
\]
we had 
\[
P_{q^2-q-1}  \xrightarrow{m}  Q_2 \xrightarrow{l_2} \mathcal{O} \xrightarrow{l_3} Q_3
\]
we get a $(q^2+q-2)$-cycle in $\Pi$. \\
Similarly, now notice that if, in $\mathcal{C}$, instead of the subpath 
\[
P_{q^2-q-1}  \xrightarrow{m}  Q_2 \xrightarrow{l_2} (2) \xrightarrow{\ell_{\infty}} (3)
\]
we had 
\[
P_{q^2-q-1}  \xrightarrow{m}   (3)
\]
we get a $(q^2+q-3)$-cycle in $\Pi$ that does not use $\ell_{\infty}$. We call this cycle $\mathcal{C}'$. \\ 
Next, for any $i=4, \cdots, q$, we delete the path 
\[
(3) \xrightarrow{l_3} Q_3   \rightarrow  \cdots  \rightarrow (i) \xrightarrow{l_i} Q_{i} 
\]
from $\mathcal{C}'$ (keeping $(3)$ and $Q_i$ in $\mathcal{C}'$) and we connect both $(3)$ and $Q_i$ with $\mathcal{O}$, using $l_3$ and $l_i$, to get the cycle
\[
(3) \xrightarrow{l_3}  \mathcal{O} \xrightarrow{l_i}     \underbrace{Q_i  \rightarrow \cdots \rightarrow P_{q^2-q-1}}_{in \ \ \mathcal{C}'} \xrightarrow{m} (3)
\]
which has length $(q^2+q-3)-(2i-6)$. Since $i=4 \cdots, q$ this yields cycles of lengths $q^2+q-5, q^2+q-7, \cdots , q^2-q+3$.\\
Now, for each of these cycles (for each $i=4, \cdots, q$), replace
\[
(3) \xrightarrow{l_3}  \mathcal{O} \xrightarrow{l_i}   Q_i
\]
by 
\[
(3) \xrightarrow{\ell_{\infty}}  (2) \xrightarrow{l_{2}} O \xrightarrow{l_i}  Q_i
\]
to get cycles with length $q^2+q-4, \cdots , q^2-q+4$. \\
Hence, the only cycle left to be constructed would be a $q^2$-cycle, in the case that $q=3$. But,  this case is easy to handle by hand, with no need of using the arguments used in this proof.
\end{proof}

\section{Wheels and Gears}

The graphs studied in this section are all, in some way, related to cycles.  Embeddings of these graphs will often rely on first finding an embedding of a specific cycle and then embedding the additional vertices and edges that make up the graph.  Most of the cycles we will be interested in will be `short' but will have some extra desirable properties.  We will focus on results for projective planes although many of the constructions are generalizable to affine planes. 

Throughout this section, we will use the same notation in the previous section. We will use $\pi_q$ to denote a generic projective plane of order $q$. 

\vspace{.1in}

\subsection{Wheel graphs}\hspace{1in}

\noindent We define the wheel graph $W_n$ to be the graph on $n+1$ vertices formed by a cycle of length $n$ and one additional vertex, called the `center', that is adjacent to every vertex in the cycle.  Hence, the center of the wheel has degree $n$.  Since no vertex of a graph embedded in $\pi_q$ can contain a vertex of degree greater than $q+1$, we see immediately that $n \leq q+1$.  We now show by construction that $W_{q+1}$ can indeed by embedded in $\pi_q$.  

\begin{theorem}\label{thmwheels}
Let $\pi_q$ be a projective plane of order $q$. Then,  $W_n   \hookrightarrow  \pi_q$ if and only if $3\leq n \leq q+1$.
\end{theorem}

\begin{proof}
Since having $W_n   \hookrightarrow  \pi_q$ implies $n \leq q+1$, we proceed to construct wheels for all $n \leq q+1$. 

We know (see \cite{LMV09}) that $K_{q+1} \hookrightarrow  \pi_q$, when $q$ is odd, and $K_{q+2} \hookrightarrow  \pi_q$, when $q$ is even. This implies that our result is obtained except, maybe,  when $n=q+1$ and $q$ is odd. 

Assume $q$ is odd, and  let $\mathcal{O}$ be any point of $\pi_q$, and $\ell=\{P_1, \cdots, P_{q+1}\}$ be any line in $\pi_q$ not incident with $\mathcal{O}$. Denote by $\ell_i$ the line joining $\mathcal{O}$ and $P_i$.

The points $P_1,  P_3, P_5, \cdots , P_q$ are vertices on the $(q+1)$-cycle of our wheel. The edges connecting these points and $\mathcal{O}$ are the corresponding $\ell_i$'s. The other $(q+1)/2$ vertices will be taken from the other $\ell_{i}$'s. Firstly, we let $m$ be a line through $P_1$, different from $\ell$ and $\ell_1$. Now choose $Q_{2i}$ to be the point  on $\ell_{2i} \cap m$, for $i=1, \cdots , (q+1)/2$. Note that none of the $Q_i$'s can be on $\ell$. 

In order to create the path
\[
P_1 \rightarrow Q_2 \rightarrow  P_3 \rightarrow  Q_4 \rightarrow  \cdots \rightarrow  P_{q} 
\]
we need to show that the lines connecting $P_i$ with $Q_{i+1 \bmod q+1}$ and $Q_j$ with $P_{j+1 \bmod q+1}$ are all distinct. But this is clear because if the line connecting $P_i$ and $Q_{i+1 \bmod q+1}$ were equal to the line connecting $P_j$ and $Q_{j+1 \bmod q+1}$ then $P_i$ and $P_j$ would be on this line, and thus either $i=j$ which is a trivial case, or this line must be $\ell$, but $\ell$ does not contain any of the $Q_i$s. A similar argument shows that all the lines needed in this path are distinct. Now, if we wanted to extend this path into a cycle by joining $P_q$ with $Q_{q+1}$, as done above, we would run into the problem of having the line joining $Q_{q+1}$ with $P_1$ being $m$, which has already been used. So, we choose a point $T_{q+1} \in \ell_{q+1}$ such that the line $t$, through $T_{q+1}$ and $P_1$, is different from $m$ and $\ell$. Since every line used in the path above goes through a $P_i$ then the lines $t$ and $s$ (joining $T_{q+1}$ and $P_q$) are different from all others. We get the $(q+1)$-cycle
\[
P_1 \rightarrow Q_2 \rightarrow  P_3 \rightarrow  Q_4 \rightarrow  \cdots \rightarrow  P_{q} \xrightarrow{s}  T_{q+1} \xrightarrow{t} P_1
\]
Since, clearly all the vertices in the cycle connect to $\mathcal{O}$ by using different lines we get that the vertices $\{ P_1,  Q_2, P_3, Q_4, \cdots , P_{q}, Q_{q+1}, \mathcal{O}\}$  form an embedding of $W_{q+1}$, as desired.
\end{proof}

\vspace{.1in}

\subsection{Gear graphs} \hspace{1in}

\noindent  We define a \emph{gear} graph, $G_n$, to be a graph on $2n+1$ vertices and $3n$ edges.  The graph consists of a $2n$-cycle, and a `center' vertex that is adjacent to every other vertex in the $2n$-cycle.  Note that no gear graph can embed in $\pi_2$, since the smallest gear graph has $9$ edges and there are only $7$ lines in such plane.  For $q=3$ and $q=4$, the only possible embeddings are $G_3 \hookrightarrow \pi_3$, $G_3 \hookrightarrow \pi_4$, $G_4 \hookrightarrow \pi_4$, and $G_5 \hookrightarrow \pi_4$. These are all easy to check, and so no details will be provided in this article. From now on we assume $q>4$.

Since the center of $G_n$ has degree $n$ then we want to prove that $G_n \hookrightarrow \pi_q$, for all $3\leq n \leq q+1$, as long as $q>4$.

\begin{lemma}\label{lemmhalfofthegears}
$G_n \hookrightarrow \pi_q$, for all $3\leq n \leq \lfloor \frac{q+1}{2}\rfloor$.
\end{lemma}

\begin{proof}
Note that if $n$ is even, then $G_{n/2}$ is a subgraph of $W_n$. The result follows from the fact that $W_n   \hookrightarrow  \pi_q$ for all $3\leq n \leq q+1$ (Theorem \ref{thmwheels}).
\end{proof}

In order to embed larger gears in $\pi_q$ we will need to construct a very specific family of cycles, and then create the appropriate gears. This is all described in the proof of our next result.  

\begin{theorem}\label{mgear}
$G_n\hookrightarrow \pi_q$ for all $3\leq n \leq q$.
\end{theorem}

\begin{proof} 
The discussion for $q=2,3,4$ was settled at the beginning of this subsection. For $q>4$,  Lemma \ref{lemmhalfofthegears} proves the theorem for all $3\leq n \leq \lfloor \frac{q+1}{2}\rfloor$. Also, it is easy to check that $G_4\hookrightarrow \pi_5$. For $n > \lfloor \frac{q+1}{2}\rfloor$ we will give explicit constructions.  Firstly, recall that two distinct paths
\[
\mathcal{P}_{P_0} : \ \ P_0 \rightarrow P_1 \rightarrow \cdots  \rightarrow P_{q} \hspace{1in} \mathcal{P}_{Q_0} : \ \ Q_1 \rightarrow Q_2 \rightarrow \cdots  \rightarrow Q_{q}
\]
constructed as described at the beginning of Section 2,  are disjoint in terms of both points and lines as long as $P_0\neq Q_0$ (two points different from $\mathcal{O}$ on $l_0$).  \\
From now on  let us fix $P_0$ and, given $n>4$, we will choose an appropriate $Q_0$ that will allow us to create a $(2n)$-cycle out of $\mathcal{P}_{P_0}$ and $\mathcal{P}_{Q_0}$. \\
Let $n$ be even. We shorten $\mathcal{P}_{P_0}$ to 
\[
\mathcal{P}_{P_0}': \ \ P_0   \rightarrow P_1   \rightarrow \cdots   \rightarrow    P_{n-2}
\]
Note that, since $n-2\leq q-2$, no lines parallel to $l_0$ or $l_1$ have been used in the construction of this path. Since $q>3$, there are at least $2$ lines through $(0)$ different from $l_0+P_{n-2}$,  $l_0+P_{1}$, and $\ell_{\infty}$, each of these lines intersect $l_1$ at a point different from $P_1$. Now, there are $q-1$ lines  through $(1)$ different from $l_1+P_{0}$ and $\ell_{\infty}$, each of these lines intersect $l_{n-1}$ at a point different from $l_{n-1}\cap(l_1+P_{0})$. We choose $Q_0\in l_0$ so that (the points on  $\mathcal{P}_{Q_0}$) $Q_{n-1} \neq l_{n-1}\cap(l_1+P_{0})$ and $Q_2 \neq P_2$ and $Q_2 \notin l_0+P_{n-2}$. We get the following $(2n)$-cycle. 
\[
(1) \xrightarrow{l_1+P_0}  \underbrace{P_0  \rightarrow \cdots \rightarrow P_{n-2}}_{in \ \ \mathcal{P}_{P_0}} \xrightarrow{l_0+P_{n-2}} (0) \xrightarrow{l_0+Q_{1}}   \underbrace{Q_1  \rightarrow \cdots \rightarrow Q_{n-1}}_{in \ \ \mathcal{P}_{Q_0}} \xrightarrow{l_1+Q_{n-1}} (1)
\]
In order to create $G_n$ we join $\mathcal{O}$ with $P_0, P_2, \cdots , P_{n-2}, Q_1, Q_3, \cdots , Q_{n-1}$. \\

For when $n$ is odd we proceed similarly.  We consider the path 
\[
\mathcal{P}_{P_0}': \ \ P_0   \rightarrow P_1   \rightarrow \cdots   \rightarrow    P_{n-1}
\]
Notice that no lines through $(0)$ or $(1)$ have been used. \\
Note that $P_{n-3}$ must be different from one of  $(l_1+P_0)\cap l_{n-3}$ and $(l_n+P_0)\cap l_{n-3}$. We will say that $P_{n-3} \neq (l_k+P_0)\cap l_{n-3}$, where $k$ is either $0$ or $n$.  So, since $P_{n-3} \neq (l_k+P_0)\cap l_{n-3}$ we choose $Q_0$ so that $Q_{n-3} = (l_k+P_0)\cap l_{n-3}$. Now notice that no two `consecutive' lines in neither $\mathcal{P}_{P_0}$ nor $\mathcal{P}_{Q_0}$ can go through $Q_0$. Hence, also considering $l_0$, there are at least two lines through $Q_0$ that have not been used so far. Let $T$ be a point on $l_n$ such that $T$ is on one of the lines through $Q_0$ that are available, and $T \neq (l_0+P_{n-1}) \cap l_n$. Note that $m=\overleftrightarrow{Q_0T}$ cannot be equal to the lines $l_k+P_{0}$, or $l_0+T$.  We get the following $(2n)$-cycle
\[
Q_{n-3}  \xrightarrow{l_k+P_{0}}  \underbrace{P_0  \rightarrow \cdots \rightarrow P_{n-1}}_{in \ \ \mathcal{P}_{P_0}} \xrightarrow{l_0+P_{n-1}} (0) \xrightarrow{l_0+T} T \xrightarrow{m}  \underbrace{Q_0  \rightarrow \cdots \rightarrow Q_{n-3}}_{in \ \ \mathcal{P}_{Q_0}}
\]
Note that $l_0+T$ could be $\ell_{\infty}$. \\
In order to create $G_n$ we join $\mathcal{O}$ with $P_0, P_2, \cdots , P_{n-1}, T, Q_1, \cdots , Q_{n-4}$. \\
Note that we used $n>4$ in order to get that $1\leq n-4$, and thus that the last selection of vertices (connected with $\mathcal{O}$) made sense.
\end{proof}

Thus far we know we can embed gear graphs, $G_n$, where $n$ is any integer between $3$ and $q$.  The next step in embeddings of gear graphs is to determine the largest that can be embedded. 

\begin{theorem}\label{maxgear}
Let $q> 4$.  Then $G_{q+1} \hookrightarrow \pi_q$.  Furthermore, this is the largest gear that can be embedded in $\pi_q$.  
\end{theorem}

\begin{proof} 
First, we notice that $G_{q+1}$ is the largest possible gear that can be embedded in $\pi_q$ because the degree of the center of $G_{q+1}$ is $q+1$, which is the largest allowed in $\pi_q$. \\
Now to show that $G_{q+1}$ actually embeds in $\pi_q$.  We need to construct a cycle of length $2(q+1)$ without using $\mathcal{O}$ and any of the $q+1$ lines through $\mathcal{O}$.  We need the cycle constructed in such a way that we are able to connect every other vertex of the cycle to the point $\mathcal{O}$. The construction of this cycle depends on the parity of the order of the plane.\\
\textbf{Case 1 ($q$ is even):} Assume $q$ is even. We will choose $q+1$ points $P_i\in l_i\setminus \{\mathcal{O}, (i)\}$, for all $i=0,1,\cdots , q$. Choose $P_1$ arbitrarily, next choose $P_3$ such that $P_3\notin l_2+P_1$, next choose $P_5$ such that $P_5\notin l_4+P_3$, etc. In general, choose $P_{i}\notin l_{i-1}+P_{i-2}$, for all $i=1,3,\cdots , q-1$ (only for $i$ odd). Now we choose $P_0$ such that $P_{0}\notin l_q+P_{q-1}$. As done before, we choose $P_{i}\notin l_{i-1}+P_{i-2}$, for all $i=2,3,\cdots , q$ (now only for $i$ even).  If for the chosen $P_q$ we get that $P_1\in l_0+P_q$ then we choose a different $P_q$. We are allowed to do this because the only condition to choose $P_q$ was that $P_{q}\notin l_{q-1}+P_{q-2}$. It follows that we get the $(2q+2)$-cycle:
\[
\hspace{-.5in} (0)  \xrightarrow{l_0+P_1} P_{1}  \xrightarrow{l_2+P_1} (2) \rightarrow \cdots \rightarrow P_{q-1} \xrightarrow{l_q+P_{q-1}} (q)  \xrightarrow{l_q+P_{0}} P_0  \xrightarrow{l_1+P_{0}} (1)  \cdots
\]
\[
\hspace{1.5in}  \cdots \rightarrow (1) \xrightarrow{l_1+P_{2}} P_2 \rightarrow \cdots \rightarrow (q-1) \xrightarrow{l_{q-1}+P_q} P_q \xrightarrow{l_0+P_q} (0)
\]
We obtain $G_{q+1}$ by joining $\mathcal{O}$ with $(0), (2), (4),\cdots (q), (1), (3), \cdots, (q-1)$ using the $q+1$ lines through $\mathcal{O}$. \\
Note that this proof also works for $q=4$. \\
\textbf{Case 2 ($q$ is odd):}  Here we do a similar construction to the one for $q$ even. We choose $P_1, P_3, \cdots , P_{q}$ as above, then we start by arbitrarily choosing $P_2$, and continue in the same fashion to get $P_2, P_4, \cdots, P_{q-1}$. In case $P_1\in l_q + P_{q-1}$ we choose a different $P_{q-1}$. Hence, the line $l_q+P_1$ has not been used yet. Now we want to find a point $T$ such that $T\notin l_0+P_q$, $T\notin l_1+P_2$, $T\notin l_0$, $T\notin l_1$, and that $T\neq P_i$, for $i=1, 2, \cdots, q$. This point will be connected to $(0)$ and $(1)$ to create the cycle we need. \\
There are $q-2$ lines through $(0)$ different from $l_0+P_q$, $l_0$, and $\ell_{\infty}$. These $q-2$ lines may be used to connect $(0)$ with $q(q-2)$ distinct points of $\pi_q$. Similarly, there are $q-2$ lines through $(1)$ different from $l_1+P_2$, $l_1$, and $\ell_{\infty}$. It follows that there are at least $q(q-2) - 2(q-1) = q(q-4)+2\geq q+2$ points that can be reached simultaneously by lines through $(0)$ or $(1)$, different from the $6$ avoided lines. Of these points, at most $q$ could be a $P_i$. Hence, there are at least two possibilities to choose $T$ from. \\
We get the following $(2q+2)$-cycle:
\[
\hspace{-.5in} (1)  \xrightarrow{l_1+P_2} P_{2}  \xrightarrow{l_3+P_2} (3) \rightarrow \cdots \rightarrow P_{q-1} \xrightarrow{l_q+P_{q-1}} (q)  \xrightarrow{l_q+P_{1}} P_1  \xrightarrow{l_2+P_{1}} (2)  \cdots
\]
\[
\hspace{.4in}  \cdots \rightarrow (2) \xrightarrow{l_2+P_{3}} P_3 \rightarrow \cdots \rightarrow (q-1) \xrightarrow{l_{q-1}+P_q} P_q \xrightarrow{l_0+P_q} (0)            \xrightarrow{l_0+T}        T       \xrightarrow{l_1+T} (1)
\]
We obtain $G_{q+1}$ by joining $\mathcal{O}$ with $(1), (3), \cdots, (q), (2), (4),\cdots (q-1), (0)$ using the $q+1$ lines through $\mathcal{O}$.
\end{proof}

\vspace{.1in}

The techniques used in  this article could also be used to study the embedding of other cycle-related graphs, such as helm graphs, prism graphs, etc. The results obtained are very similar to those in this article. We are currently trying to develop a theory of embeddings of graphs in finite projective planes and spaces.


\end{document}